\documentclass[a4paper]{amsart}
\usepackage{amsmath, amsthm, amssymb, fullpage}

\newtheorem{prop}{Proposition}
\newtheorem{cor}{Corollary}
\theoremstyle{remark} 
\newtheorem{rem}{Remark}
\theoremstyle{definition} 
\newtheorem{defi}{Definition}
\newcommand{\sgn}{\mathrm{sgn}}
\newcommand{\Lie}{\mathit{Lie}}
\newcommand{\Ass}{\mathit{Ass}}
\newcommand{\FreeLie}{\mathrm{FreeLie}}
\newcommand{\FreeAss}{\mathrm{FreeAss}}

\title[The cubical complex of a permutation group representation]{The cubical complex of a permutation group representation -- or however you want to call it}
\author{Pavol \v Severa}
\address{P. \v S.: University of Geneva}
\author{Thomas Willwacher}
\address{T. W.:Harvard University}

\begin{document}
\maketitle
\begin{abstract}
This paper is about a small combinatorial trick, which is well known, but has no name.
Let $G$ be a permutation group acting on a vector space $M$. There is a natural way to assign a cosimplicial space to these data. We call the resulting cochain complex the cubical complex. Its cohomology is easy to compute. We give some examples of its occurrence in nature.
\end{abstract}

\section{A general combinatorial fact}
Let $k$ be a field of characteristic zero, $n\in\mathbb{N}$ and $A=k[x_1,\dots,x_n]$ the polynomial co-algebra. Let $C(A)$ be its bar complex, computing $Tor_A(k,k)$. Since $A$ is Koszul it is well known that the homology $H(C(A))=Tor_A(k,k) \cong k[\epsilon_1,\dots, \epsilon_n]$ is the Koszul dual algebra, where $\epsilon_j$ are odd (degree one) variables.

On the complex $C(A)$ there is a $\mathbb{N}^n_0$-grading by the numbers of $x_1,\dots, x_n$ occuring. Let $C(A)^{(1,\dots,1)}$ be the degree $(1,\dots,1)$ subcomplex. For instance, if $n=3$ then $x_1\otimes x_2x_3$ is inside this subcomplex, while  $x_1\otimes x_2x_1$ is not. It follows from the above that 
\begin{prop}
\[
H(C(A)^{(1,\dots,1)})=
( k[\epsilon_1,\dots, \epsilon_n])^{(1,\dots,1)}=k\,\epsilon_1\dots \epsilon_n \cong k[-n].
\]
\end{prop}

\begin{rem}
As noted by V. Drinfeld \cite{drinfeld}, the complex $C(A)^{(1,\dots,1)}$
is the cosimplicial complex of an $n$ dimensional hypercube, relative to its boundary. 
\end{rem}

Note further that $C(A)^{(1,\dots,1)} \cong \oplus_m (k^m)^{\otimes n}$, $m$ being the degree. Note also that the differential is invariant 
under the $S_n$ action by permuting the factors. Since taking (co-)invariants under finite group actions commutes with taking homology, we obtain:

\begin{cor}
\label{cor:main}
Let $G\subset S_n$ be a subgroup and $M$ some right $G$-module. Then 
\[
H(  M \otimes_G C(A)^{(1,\dots,1)}  ) =  M \otimes_G \sgn_n
\]
where $\sgn_n$ is the one dimensional sign representation of $G$, concentrated in degree $n$.
\end{cor}

\begin{defi}
We call the complex $M \otimes_G C(A)^{(1,\dots,1)}$ the \emph{cubical complex} $\mathit{Cub}(G,M)$ of the pair $(G, M)$.
\end{defi}
\begin{rem}
There is a similar complex based on the Harrison complex $\mathit{Harr}(A)$, namely
\[
\mathit{Harr}(G,M) = M \otimes_G \mathit{Harr}(A)^{(1,\dots,1)},
\]
which we call the Harrison complex of the pair $(G,M)$.
It is well known that $H (\mathit{Harr}(A))\cong k^n[-1]$. Hence the following is also well known:
\[
H (\mathit{Harr}(G,M))\cong 
\begin{cases}
0 &\text{for $n>1$}\\
M[-1] &\text{for $n=1$}.
\end{cases}
\]
\end{rem}
\begin{rem}
It is actually sufficient to consider the case $G=S_n$. If $G\subsetneq S_n$, one can equivalently take $G'=S_n$, $M'= M\otimes_G k[S_n]=Ind_G^{S_n}M$.
\end{rem}

\section{Examples}

Let us take $G=S_n$ and $M=\Lie(n)$ the $n$-ary operations in the Lie operad. By definition, 
\begin{multline*}
\oplus_n\mathit{Cub}(S_n,\Lie(n))=
\oplus_n \Lie(n)  \otimes_{S_n}  C(A)^{(1,\dots,1)} =\\ =\oplus_{m,n}  \Lie(n)  \otimes_{S_n}  (k^m)^{\otimes n}
= \oplus_m \FreeLie(k^m) =: \mathfrak{lie}
\end{multline*}
is a complex built from the free Lie algebras, which one can check is the same as the one considered by A. Alekseev and C. Torossian \cite{AT}.
One hence obtains:
\begin{cor}
\[
H(\mathfrak{lie}) = \oplus_n \Lie(n) \otimes_{S_n} \sgn_n = k[-1] \oplus k[-2]. 
\]
\end{cor}
\begin{proof}
The first equality is the previous Corollary, second follows from the fact that there are no fully antisymmetric elements in $\Lie(n)$ for $n\geq 3$ by the Jacobi identity.
\end{proof}

This cohomology was computed by M. Vergne \cite{vergne} (the second part of Thm 2.3 in \cite{vergne}).

A simpler and very similar example is $G=S_n$ and $M=k[S_n]=\Ass(n)$, the space of $n$-ary operations in the associative operad. In this case $\oplus_n\mathit{Cub}(S_n,k[S_n])$ is a complex buit from free associative algebras and its cohomology is $\oplus_n k[S_n]\otimes_{S_n}\sgn_n=\oplus_n k[-n]$.
This is the first part of Theorem 2.3 in \cite{vergne}.

Next, take again $G=S_{n}$ and $M=k[S_{n}]_{C_n}$ where $C_n\subset S_n$ is the cyclic group. Equivalently, $M=\Ass(n-1)$, where we view $\Ass$ as a cyclic operad. In this case we get the complex $\mathfrak{tr}$ defined by Alekseev and Torossian \cite{AT}
\begin{align*}
\mathfrak{tr} &= \oplus_n k[S_n]_{C_n}\otimes_{S_n}  C(A)^{(1,\dots,1)}\cong  \oplus_{m,n} k[S_n]_{C_n}\otimes_{S_n}  (k^m)^{\otimes n} =  \oplus_{m,n}( (k^m)^{\otimes n})_{C_n} \\
&=  \oplus_{m} \FreeAss(k^m) / [ \FreeAss(k^m) ,  \FreeAss(k^m)].
\end{align*} 

Applying Corollary \ref{cor:main} we get
\begin{cor}
\[
H(\mathfrak{tr}) =  \oplus_n k[S_n]_{C_n}\otimes_{S_n}   \sgn_n =  \oplus_n  k \otimes_{C_n}  \sgn_n \cong \oplus_{n\,\text{odd}} k[-n]
\]
\end{cor}

We can also take $G=S_{n}$ and $M=\Lie(n-1)$. This is possible since $\Lie$ (-as a suboperad of $\Ass$-) is a cyclic operad. Consider 
\[
\oplus_n \Lie(n)  \otimes_{S_{n+1}}  C(A)^{(1,\dots,1)} = \oplus_{m,n}  \Lie(n)  \otimes_{S_{n+1}}  (k^m)^{\otimes n}
= \oplus_m \mathfrak{sder}_m =: \mathfrak{sder}.
\]
Note that this space coincides with the total space of the operad of Lie algebras $\mathfrak{sder}$ defined in \cite{AT}, which in turn is the same as the $\mathcal{F}$ defined in \cite{drinfeld}. By a similar argument as before, we hence obtain:

\begin{cor}
\[
H(\mathfrak{sder}) =  \oplus_n \Lie(n) \otimes_{S_{n+1}}   \sgn_{n+1} \cong k[-3].
\]
\end{cor}

\begin{rem}
All examples above have analogs for their respective Harrison complexes. Concretely, the cohomology in this case is always given by the $n=1$-part of the cohomology computed in the above propositions. We leave the details to the reader.
\end{rem}

\section{Generalizations and further examples in the literature}
There is a generalization to the above trick. Let $\mathcal P$ be any cooperad and $A$ the free $\mathcal P$-coalgebra on $k^n$. Then $A$ is endowed with a $\mathbb{Z}^n$-grading, along with an $S_n$ action. Any ``sensible'' homology theory of $\mathcal P$-coalgebras will produce a complex $C(A)$, inheriting this grading and $S_n$ action. For any right $S_n$ module $M$ one can hence form the complex
\[
M\otimes_{S_n} C(A)^{(1,\dots,1)}
\]
with cohomology $M \otimes_{S_n} H(C(A))^{(1,\dots,1)}$. This situation occurs frequently in nature, for example, taking for $\mathcal P$ the coassociative cooperad and $C(A)$ the cyclic complex, one can produce a proof of Theorem 8 in \cite{barnatan}.
Taking $P=e_2^*\{2\}$ the co-Gerstenhaber cooperad (up to degree shift), and for $C(A)$ the Gerstenhaber cohomology, one can produce a shorter proof of Proposition 21 in \cite{thomas}. 

We want to emphasize that the combinatorial trick described here is not novel, and has been used many times in the mathematical literature, in various forms and by various authors.\footnote{In particular, D. Bar-Natan made us aware of his article \cite{barnatan2}, where the case of general $M$ is discussed as well.}
This paper was written merely to advertise a clean version of the trick and point out some further applications.

\bibliographystyle{plain}	
\bibliography{thebib}

\end{document}